\documentclass[11pt,a4paper]{article}
\usepackage{amsmath}
\usepackage{amssymb}
\usepackage{amsthm}
\usepackage{amsfonts}
\usepackage{enumerate}
\usepackage{xspace} 
\usepackage{graphics} 
\usepackage{graphicx}
\usepackage{pgfpages}
\usepackage{url}
\usepackage{booktabs}
\usepackage{verbatim}
\usepackage{hyperref}
\usepackage{breakurl}
\usepackage[margin=1.3in]{geometry}
\newcommand{\tluste}[1]{\mbox{\mathversion{bold}$ #1 $}}

\newcommand{\mna}[1]{{\mathcal{#1}}}
\newcommand{\imace}[1]{\mbox{$\tluste{#1}$}}

\newcommand{\onum}[1]{\mbox{$\overline{{#1}}$}} 
\newcommand{\unum}[1]{\mbox{$\underline{{#1}}$}}

\newcommand{\R}[0]{{\mathbb{R}}}

\def\eps{{\varepsilon}}
\newcommand{\mmid}[0]{;\,}		

\newcommand{\seznam}[1]{{\{1, \ldots, {#1}\}}}
\newcommand{\sseznam}[2]{{\{{#1}, \ldots, {#2}\}}}
\def\clqq{``}
\def\crqq{''}
\def\quo#1{\clqq{}#1\crqq{}}  

\DeclareMathOperator{\rr}{r}

\def\nref#1{$(\ref{#1})$}

\newtheorem{theorem}{Theorem}

\newtheorem{proposition}{Proposition}

\newtheorem{corollary}{Corollary}

\theoremstyle{definition}

\newtheorem{example}{Example}
\newtheorem{remark}{Remark}

\begin{document}

\title{Tolerances, robustness and parametrization of matrix properties related to optimization problems}

\author{
  Milan Hlad\'{i}k\footnote{
Charles University, Faculty  of  Mathematics  and  Physics,
Department of Applied Mathematics, 
Malostransk\'e n\'am.~25, 11800, Prague, Czech Republic, 
e-mail: \texttt{milan.hladik@matfyz.cz}
}
}

\date{\today}
\maketitle

\begin{abstract}
When we speak about parametric programming, sensitivity analysis, or related topics, we usually mean the problem of studying specified perturbations of the data such that for a given optimization problem some optimality criterion remains satisfied. In this paper, we turn to another question. Suppose that $A$ is a matrix having a specific property $\mathcal{P}$. What are the maximal allowable variations of the data such that the property still remains valid for the matrix? We study two basic forms of perturbations. The first is a perturbation in a given direction, which is closely related to parametric programming. The second type consists of all possible data variations in a neighbourhood specified by a certain matrix norm; this is related to the tolerance approach to sensitivity analysis, or to stability. The matrix properties discussed in this paper are positive definiteness; P-matrix, H-matrix and P-matrix property; total positivity; inverse M-matrix property and inverse nonnegativity.
\end{abstract}


\section{Introduction}


So far, parametric programming and sensitivity analysis in mathematical programming \cite{Gal1979,GalGre1997,NozGud1974} was mainly focused on studying optimality criteria under certain data perturbations. In contrast, herein we will be devoted to parametrization and sensitivity of special matrix properties. Notice that many matrix properties closely relate to properties of optimization problems. For example, positive definiteness of the Hessian matrix indicates convexity of a function, or P-matrix property shows unique solvability of a linear complementarity problem. Stability and sensitivity of such matrix properties therefore reflect stability and sensitivity of the corresponding optimization problem.

\paragraph{Notation.}
The all-ones vector of convenient length is denoted by $e:=(1,\dots,1)^T$, the all-ones matrix by $E:=ee^T$, the identity matrix of size $n$ by $I_n$, and the $i$th standard unit vector by $e_i$.
For a matrix $A$, we use $A^{ij}$ for the matrix obtained from $A$ by removing the $i$th row and $j$th column. The spectral radius of $A$ is $\rho(A)$.

\paragraph{Vector and matrix norms.}

Among the vector norms, we will particularly utilize $p$-norms defined for every $p\geq 1$ and $x\in\R^n$ as $\|x\|_p:=\big(\sum_{i=1}^n|x|_i^p\big)^\frac{1}{p}$. Subsequently, the maximum norm reads $\|x\|_{\infty}=\max_i|x|_i$.

For two arbitrary vector norms $\|x\|_{\alpha},\|x\|_{\beta}$, the subordinate matrix norm \cite{Hig1996} is defined by 
$$
\|A\|_{\alpha,\beta}:=\max_{\|x\|_{\alpha}=1}\|x\|_{\beta}.
$$
Utilizing the vector maximum norm and $1$-norm, we get
\begin{align}\label{normInf1}
\|A\|_{\infty,1}
:=\max_{\|x\|_{\infty}=1}\|x\|_{1}
 =\max_{y,z\in\{\pm1\}^n} y^TAz;
\end{align}
see \cite{Fie2006}. In contrast to many other norms, this one is NP-hard to compute \cite{Roh2000}.
Provided $\alpha\equiv\beta$, the subordinate matrix norm reduces to the standard induced matrix norm. In particular, the 2-norm (or, the spectral norm) of $A$ is equal to the maximum singular value, that is $\|A\|_2=\sigma_{\max}(A)$. Another frequently used (non-induced) matrix norm is the Frobenius norm $\|A\|_F:=\sqrt{\sum_{i,j}a_{ij}^2}$.

A matrix norm is \emph{consistent} if $\|AB\|\leq\|A\|\cdot\|B\|$ for every $A,B\in\R^{n\times n}$. For consistent matrix norms we have for each matrix $A\in\R^{n\times n}$ the upper bound on its spectral radius $\rho(A)\leq\|A\|$.
Notice that the max-norm
$$
\|A\|_{\max}:=\max_{i,j}|a_{ij}|
$$
is not consistent. In contrast, any induced matrix norm is consistent 
and satisfies 
\begin{align}\label{normI}
\|I_n\|=1. 
\end{align}
This property holds also for some other norms, e.g., the max-norm, but not for example for the Frobenius norm.

We will particularly utilize matrix norms satisfying
\begin{subequations}\label{normAss}
\begin{align}
\label{normAss1}
\|A'\| &\leq \|A\|\quad\mbox{whenever $A'$ is a submatrix of $A$},\\
\label{normAss2}
\|e_i^{}e_j^T\|&=1\quad \forall i,j.
\end{align}
\end{subequations}
These two properties are not very restrictive since most of the commonly used norms fulfill them. They are satisfied by any induced $p$-norm, Frobenius norm or max-norm, for instance. 

A matrix norm is called \emph{absolute} if $\|A\|=\||A|\|$. This property is satisfied for the induced $1$-~and $\infty$-norm, Frobenius norm or max-norm, for instance, but not for the spectral norm. 
A matrix norm is \emph{monotone} if $|A|\leq B$ implies $\|A\|\leq\|B\|$. This holds for any induced $p$-norm, Frobenius norm or max-norm, for instance.

\paragraph{Regularity radius.}
Let $A\in\R^{n\times n}$ be a nonsingular matrix and $\|\cdot\|$ a selected matrix norm. Regularity radius is the distance to the nearest singular matrix and denoted as
$$
\rr(A):=\min\{\|A-B\|\mmid B\mbox{ is singular}\}.
$$
For the spectral norm, Frobenius norm and some other orthogonally invariant matrix norms, the regularity radius is given by the smallest singular value, so $\rr(A)=\sigma_{\min}(A)$.
More in general, for any induced matrix norm, the regularity radius is described by the formula $\rr(A)=\|A^{-1}\|^{-1}$ by the Gastinel--Kahan theorem \cite{Hig1996,Kah1966}.

For the max-norm, the regularity radius can be expressed as \cite{HarHla2016a,KreLak1998,PolRoh1993} 
\begin{align}\label{radReg}
\rr(A)=\frac{1}{\max_{y,z\in\{\pm1\}^n} y^TA^{-1}z}.
\end{align}
Utilizing the matrix norm \nref{normInf1}, the above value reads $\|A^{-1}\|_{\infty,1}^{-1}$; see \cite{Fie2006}. For this reason, determining the regularity radius $\rr(A)$ is NP-hard \cite{PolRoh1993}.

\paragraph{Special matrices and the structure of the paper.}
Let us introduce some of the special matrices addressed in this paper. A nonsingular matrix $A\in\R^{n\times n}$ is \emph{inverse nonnegative} if $A^{-1}\geq0$. The inequality between vectors and matrices in understood entrywise throughout the paper.
A matrix $A\in\R^{n\times n}$ is an \emph{M-matrix} if $a_{ij}\leq0$ for every $i\not=j$ and one of the following equivalent conditions holds \cite{HorJoh1991}
\begin{itemize}
\item
$A^{-1}\geq0$,
\item
$Av>0$ for some $v>0$.
\item
$A=k I_n-P$, where $P\geq0$ and $k>\rho(P)$,
\item
all eigenvalues are positive,
\item
real parts of all eigenvalues are positive.
\end{itemize}
A matrix $A\in\R^{n\times n}$ is an \emph{H-matrix} if the so called \emph{comparison matrix} $\langle A\rangle$ is an M-matrix, where $\langle A\rangle_{ii}=|a_{ii}|$ and $\langle A\rangle_{ij}=-|a_{ij}|$ for $i\not=j$.
Other types of matrices will be introduced later.

For a particular matrix $A\in\R^{n\times n}$ and property $\mna{P}$, we investigate two problems: its parametrization and tolerance radius. The former represents a perturbation in a specified direction, whereas the latter considers all perturbations within a particular neighbourhood.

By the \emph{parametrization} of a matrix $A\in\R^{n\times n}$, we mean a matrix $A_{\delta}=A-\delta \tilde{A}$, where $\delta$ is a parameter and $\tilde{A}\in\R^{n\times n}$ is given. We are interested in the set of all admissible values, that is, the values of $\delta$ such that $A_{\delta}$ satisfies property  $\mna{P}$. Since the set of admissible values can be complicated, we restrict often ourselves to computing a (closed, open or semi-open) interval $\mathfrak{o}\ni0$ such that each $\delta\in\mathfrak{o}$ is admissible.

The \emph{tolerance radius} is usually defined as follows
\begin{align*}
\delta^*&=\inf\{\delta\geq0\mmid 
 \exists A': \|A'\|\leq\delta, A+A'\mbox{ does not satisfy }\mna{P}\},\\
 &=\sup\{\delta\geq0\mmid 
  A+A'\mbox{ satisfies }\mna{P}\mbox{ for all }A':\|A'\|<\delta\}.
\end{align*}
That is, $\delta^*$ is the minimum distance to a matrix not satisfying property $\mna{P}$ in a given matrix norm.
In some cases, we will consider structured perturbation matrices $A'$ (for example, symmetric perturbations for positive definiteness).

In the case of max-norm, the tolerance radius is closely related to tolerance analysis. Tolerance analysis asks for maximal simultaneous and independent variations of input coefficients such that some invariant remains valid. In the case of sensitivity analysis in linear programming, the tolerance approach uses the invariant representing optimality of a basis or optimal partition \cite{Hla2011c,Wen1985,WenChen2010}. In the case of linear systems of equations and inequalities, the invariant may represent also its (in)feasibility \cite{HlaRoh2016a}, among others.

\section{Positive determinant}\label{sDet}

We investigate positive determinant first. This is maybe not the most interesting property itself, but will utilize it later on for computing tolerances and stability radii of other properties. Throughout this section, let us have $A\in\R^{n\times n}$ be such that $\det(A)>0$.

\paragraph{Parametrization.} 
Consider a parametrized matrix $A$ in the form $A_{\delta}=A-\delta \tilde{A}$, where $\tilde{A}\in\R^{n\times n}$ is given. We are interested in the set of all admissible values of $\delta$ for which $\det(A_{\delta})>0$. Expansion of the determinant $\det(A_{\delta})=\det(A-\delta \tilde{A})$ yields a polynomial of degree at most $n$, from which the range could be computed. Hence the set of all admissible values of $\delta$ is formed by union of at most $\lfloor\frac{n}{2}\rfloor$ open intervals, not necessarily bounded.

Provided $\tilde{A}$ has rank one, the range of admissible values is more easy to find. In this case $\tilde{A}$ has the form of $\tilde{A}=ab^T$ for some $a,b\in\R^n$.

\begin{theorem}\label{thmDetParRankone}
If $\tilde{A}=ab^T$, then the set of admissible values of $\delta$ for $A_{\delta}$ to have positive determinant is 
\begin{itemize}
\item
$(-\infty,\frac{1}{b^TA^{-1}a})$ provided $b^TA^{-1}a>0$,
\item
$(\frac{1}{b^TA^{-1}a},\infty)$ provided $b^TA^{-1}a<0$,
\item
$\R$ provided $b^TA^{-1}a=0$.
\end{itemize}
\end{theorem}

\begin{proof}
By the Sherman--Morrison formula 
\begin{align*}
\det(A_{\delta})
=\det(A-\delta ab^T)
=\det(A)(1-\delta b^TA^{-1}a).
\end{align*}
Thus we have the constraint $1>\delta b^TA^{-1}a$, from which $\delta$ is easily derived.
\end{proof}

\paragraph{Radius of positive determinant.} 
Define the radius of positive determinant of $A$ as
\begin{align*}
\delta^*:=  \sup\{\delta\geq0 \mmid 
  \det(A+A')>0\ \forall A': \|A'\|<\delta\},
\end{align*}
where $\|\cdot\|$ is a given matrix norm. It turns out that $\delta^*$ is exactly the radius of nonsingularity $\rr(A)$.

\begin{theorem}\label{thmDetRad}
We have $\delta^*=\rr(A)$.
\end{theorem}

\begin{proof}
Case \quo{$\delta^*\leq\rr(A)$}.
If $\det(A+A')>0$, then $A+A'$ is nonsingular.

Case \quo{$\delta^*\geq\rr(A)$}.
If $\det(A+A')\leq0$, then $\det(A+\alpha A')=0$ for some $\alpha\in(0,1]$, whence $A+\alpha A'$ is singular.
\end{proof}

\section{Positive definiteness}\label{sPd}

In this section we suppose that $A\in\R^{n\times n}$ is (symmetric) positive definite.

\paragraph{Parametrization of positive definiteness.}

Consider the parametrized matrix $A$ in the form $A_{\delta}=A-\delta \tilde{A}$, where $\tilde{A}\in\R^{n\times n}$ is symmetric. The aim is to compute the range of admissible values of $\delta$ for which $A_{\delta}$ remains positive definite.  
Due to continuity of eigenvalues, the range of admissible values of $\delta$ for positive definiteness is the same as for positive determinant. Thus, the problem reduces to computing admissible values for $\det(A_{\delta})>0$.

Focus now on rank one parametrization in the form $A_{\delta}=A-\delta aa^T$, where $a\in\R^n\setminus\{0\}$ is a given vector. 
Obviously,  $A_{\delta}$ is positive definite for every $\delta\leq0$, so we will focus on computing the upper bound only
\begin{align*}
\delta^*:=\sup\{\delta\geq0 \mmid 
  A-\delta aa^T\mbox{ is positive definite}\}.
\end{align*}
Provided $a$ is an eigenvector of $A$ corresponding to an eigenvalue $\lambda$, then $A-\delta aa^T$ decreases this eigenvalue by $\delta$, so we have $\delta^*=\lambda$. In general, the following holds:

\begin{theorem}\label{thmPsdPar}
We have $\delta^*=\frac{1}{a^TA^{-1}a}$.
\end{theorem}

\begin{proof}
It follows from Theorem~\ref{thmDetParRankone}, first item, since $a^TA^{-1}a>0$.
\end{proof}

\paragraph{Radius of positive definiteness.}

Define the radius of positive definiteness of $A$ as the minimum distance to a symmetric non-positive definite matrix as follows
\begin{align*}
\delta^*:=  \sup\{\delta\geq0 \mmid 
  A+A'\mbox{ is positive definite }\forall A': A'=A'^T,\ \|A'\|<\delta\},
\end{align*}
where $\|\cdot\|$ is a given matrix norm.

\begin{theorem}
For every consistent matrix norm satisfying \nref{normI} we have $\delta^*=\lambda_{\min}(A)$.
\end{theorem}

\begin{proof}
Let $A'$ be any symmetric matrix such that $\|A'\|<\lambda_{\min}(A)$. Then by Weyl's theorem \cite{HorJoh1985}
\begin{align*}
\lambda_{\min}(A+A')
\geq \lambda_{\min}(A)+\lambda_{\min}(A')
\geq \lambda_{\min}(A)-\rho(A')
\geq \lambda_{\min}(A)-\|A'\|
> 0.
\end{align*}
Hence $A+A'$ is positive definite.
On the other hand, for the matrix $A':=-\lambda_{\min}(A) I_n$ with the norm $\|A'\|=\lambda_{\min}(A)$ we have $\lambda_{\min}(A+A')=0$, so the radius of positive definiteness cannot be larger than $\lambda_{\min}(A)$. 
\end{proof}

Considering the max-norm, the situation is worse. We have $\delta^*>1$ if and only if the interval matrix $\imace{A}=[A-E,A+E]$ is positive definite, that is, every $A\in\imace{A}$ is positive definite. Checking this property is, however, co-NP-hard \cite{KreLak1998,Roh1994}.
The minimum eigenvalue $\lambda_{\min}(\imace{A}_{\delta})$ of $\imace{A}_{\delta}=[A-\delta E,A+\delta E]$ is defined the minimum eigenvalue on the set of symmetric matrices in $\imace{A}_{\delta}$. It can be expressed by the Hertz formula \cite{Her1992} as
$$
\lambda_{\min}(\imace{A}_{\delta})
 =\min_{y\in\{\pm1\}^n}\lambda_{\min}(A-\delta yy^T).
$$
Thus the radius of positive definiteness can be formulated as minimal $\delta$ such that matrices $A-\delta yy^T$ are positive definite for every $y\in\{\pm1\}^n$.
Based on Theorem~\ref{thmPsdPar}, we have the following. 

\begin{theorem}
For the max-norm, we have $\delta^*=\min_{y\in\{\pm1\}^n}\frac{1}{y^TA^{-1}y}$.
\end{theorem}

The above formula can be equivalently expressed as
\begin{align}\label{radPsdYy}
\delta^*=\frac{1}{\max_{y\in\{\pm1\}^n} y^TA^{-1}y}.
\end{align}
This is in an interesting correspondence with a regularity radius of a nonsingular matrix $B\in\R^{n\times n}$, which by \nref{radReg} reads
\begin{align*}
\rr(B)
=\|B^{-1}\|_{\infty,1}^{-1}
=\frac{1}{\max_{y,z\in\{\pm1\}^n} y^TB^{-1}z}.
\end{align*}
Since $A^{-1}$ is positive definite, it has the positive definite square root $\sqrt{A^{-1}}$. Thus we get 
$
y^TA^{-1}y
=y^T\sqrt{A^{-1}}\sqrt{A^{-1}}y
=\|\sqrt{A^{-1}}y\|_2^2
$
and
$$
\max_{y\in\{\pm1\}^n} y^TA^{-1}y
=\max_{y\in\{\pm1\}^n} \|\sqrt{A^{-1}}y\|_2^2
=\max_{\|y\|_{\infty}=1} \|\sqrt{A^{-1}}y\|_2^2.
$$
Hence \nref{radPsdYy} has equivalent form
\begin{align*}
\delta^*=\frac{1}{\|\sqrt{A^{-1}}\|_{\infty,2}^2}.
\end{align*}

As a consequence of NP-hardness of computing the $\|\cdot\|_{\infty,1}$ norm we have 

\begin{corollary}
Computing the norm $\|B\|_{\infty,2}$ is an NP-hard problem.
\end{corollary}

Due to intractability of determining $\delta^*$ for the max-norm, polynomially computable cases and simple bounds can be useful. 
A tractable case is when $A$ is inverse nonnegative simply because $\max_{y\in\{\pm1\}^n} y^TA^{-1}y=e^TA^{-1}e$. Recall that M-matrices belong to the class of inverse nonnegative matrices for example.

\begin{proposition}
If $A$ is inverse nonnegative, then for the max-norm we have $\delta^*=\frac{1}{e^TA^{-1}e}$.
\end{proposition}

\begin{proposition}
For the max-norm we have $\delta^*\geq \frac{1}{n}\lambda_{\min}(A)$.
\end{proposition}

\begin{proof}
Each eigenvalue of every symmetric $\tilde{A}\in\imace{A}_{\delta}$ is bounded from below by $\lambda_{\min}(A)-\lambda_{\max}(\delta E)=\lambda_{\min}(A)-\delta n$; see \cite{HlaDan2010,Roh2012a}. Therefore, if $\lambda_{\min}(A)-\delta n>0$, then $\tilde{A}$ is positive definite.
\end{proof}

Let $A$ be a positive definite M-matrix. By \cite{Hla2017ua}, the interval matrix $\imace{A}=[A-E,A+E]$ is positive definite if and only if it is an H-matrix. Thus, the radius of positive definiteness of this $A$ can be handled by techniques from Section~\ref{sHmat}.

\section{P-matrix property}

A square matrix is a P-matrix if all its principal minors are positive. 
P-matrices play an important role in linear complementarity problems \cite{CotPan2009,Mur1988,Scha2004}
$$
q+Mx\geq0,\ \ x\geq0,\ \ (q+Mx)^Tx=0,
$$
which appear in so many situations such a quadratic programming, bimatrix games, or equilibria in specific economies. Such a complementarity problem has a unique solution for each $q$ if and only if $M$ is a P-matrix.

The problem of checking whether a given matrix is a P-matrix is known to be co-NP-hard \cite{Cox1994,KreLak1998}. There are, however, efficiently recognizable subclasses such as positive definite matrices (discussed in Section~\ref{sPd}), M-matrices (Section~\ref{sMmat}), H-matrices with positive diagonal entries (Section~\ref{sHmat}), or totally positive matrices (Section~\ref{sTp}).

In the past, there was a research also in the error bounds of the solutions of linear complementarity problems \cite{ChenXia2006}, but we will be concerned with error bounds for the matrix itself.

\paragraph{Parametrization.}
Consider the parametrization of a P-matrix $A$ in the form $A_{\delta}=A-\delta \tilde{A}$, where $\tilde{A}\in\R^{n\times n}$ is given.
To find the range of values, for which $A_{\delta}$ remains a P-matrix, we have to inspect all principal submatrices $\hat{A}_{\delta}$ and for each of them to find the range of $\delta$ for which $\det(\hat{A}_{\delta})>0$. 
Positive determinant was dealt with in Section~\ref{sDet}, so the problem reduces to $2^n-1$ problems of positive determinant parametrization. By Theorem~\ref{thmDetParRankone}, we have straightforwardly the following result when $\tilde{A}$ has rank one.

\begin{theorem}
If $\tilde{A}=ab^T$, then the set of admissible values of $\delta$ for $A_{\delta}$ to be an P-matrix is $(\unum{\delta},\onum{\delta})$, where
\begin{align}\label{d1ThmPmatParRankone}
\unum{\delta}
 &=\sup\left\{ \frac{1}{\hat{b}^T\hat{A}^{-1}\hat{a}} \mmid
   \hat{A}\mbox{ is a principal submatrix of }A,\ 
   \hat{b}^T\hat{A}^{-1}\hat{a}<0\right\},\\
\label{d2ThmPmatParRankone}
\onum{\delta}
 &=\inf\left\{ \frac{1}{\hat{b}^T\hat{A}^{-1}\hat{a}} \mmid
   \hat{A}\mbox{ is a principal submatrix of }A,\ 
   \hat{b}^T\hat{A}^{-1}\hat{a}>0\right\},
\end{align}
and $\hat{a},\hat{b}$ denote the restrictions of $a,b$ to the corresponding subvectors compatible with~$\hat{A}$.
\end{theorem}

In some situations, the exponential number of principal submatrices to process can be decreased. One of such situations is when $A$ is an M-matrix, and $\tilde{A}\geq0$ has rank one. A similar result holds when $\tilde{A}\leq0$.

\begin{theorem}
Suppose that $A$ is an M-matrix and $\tilde{A}=ab^T\geq0$. Then the set of admissible values of $\delta$ for $A_{\delta}$ to be an P-matrix is 
\begin{itemize}
\item
$(-\infty,\frac{1}{b^TA^{-1}a})$ provided $b^TA^{-1}a>0$,
\item
$(-\infty,\infty)$ provided $b^TA^{-1}a=0$.
\end{itemize}
\end{theorem}

\begin{proof}
Denote by the hat a principal submatrix, so that $\hat{A}$ is a principal submatrix of $A$, and $\widehat{(A^{-1})}$ analogously for $A^{-1}$. By \cite{JohSmi2011} we have $0\leq \hat{A}^{-1}\leq\widehat{(A^{-1})}$.
Further, denote by $\hat{a},\hat{b}$ the corresponding subvectors of $a,b$. Now,
\begin{align*}
0 \leq \tilde{b}^T\hat{A}^{-1}\tilde{a}
  \leq \tilde{b}^T\widehat{(A^{-1})}\tilde{a}
  \leq b^TA^{-1}a.
\end{align*}
Therefore the minimum value in \nref{d2ThmPmatParRankone} is $\frac{1}{b^TA^{-1}a}$; if $b^TA^{-1}a=0$ then it is $\infty$. The maximum value in \nref{d1ThmPmatParRankone} is $-\infty$ since $b^T\hat{A}^{-1}a\geq0$ for every $\hat{A}$.
\end{proof}

\paragraph{P-matrix radius.}
For a given matrix norm, define the radius of P-matrix property of $A$ as the distance to the nearest matrix that is not a P-matrix:
\begin{align*}
\delta^*:=  \sup\{\delta\geq0 \mmid 
  A+A'\mbox{ is an P-matrix }\forall A': \|A'\|<\delta\}.
\end{align*}

For any matrix norm, we can calculate the P-matrix radius by reduction to regularity radii in the same norm of its principal submatrices. 

\begin{theorem}
For any matrix norm we have 
$$
\delta^*= \min\{\rr(\hat{A})\mmid \hat{A}\mbox{ is a principal submatrix of }A\}.
$$
\end{theorem}

\begin{proof}
The value $\rr(A)$ gives the distance to the nearest singular matrix, that is, the distance to the nearest matrix with non-positive determinant; see Theorem~\ref{thmDetRad}. This applies to  principal submatrices of $A$, too.
\end{proof}

For the spectral norm, Frobenius norm and some other orthogonally invariant matrix norms, the regularity radius is equal to given by the smallest singular value. Thus, we have as a consequence:

\begin{corollary}\label{corPmatRadSpectr}
For the spectral or Frobenius norm we have 
$$
\delta^* = \min\{\sigma_{\min}(\hat{A})\mmid
  \hat{A}\mbox{ is a principal submatrix of }A\}.
$$
\end{corollary}

Unfortunately, there is no monotonicity of smallest singular values with respect to principal submatrices, so we have to inspect all of them in general.

\begin{example}
Consider the P-matrix
$$
A=\begin{pmatrix}10&5\\-5&1\end{pmatrix}.
$$
Then $\sigma_{\min}(A)\approx2.933$, but the principal submatrix $\hat{A}=(1)$ yields a smaller value of $\sigma_{\min}(\hat{A})$ and therefore $\delta^*=1$.
Let 
$$
A=\begin{pmatrix}10&2\\2&1\end{pmatrix}.
$$
Then $\sigma_{\min}(A)\approx0.5756$, which is the smallest one over all principal submatrices and thus $\delta^*\approx0.5756$. Indeed, if we construct from the SVD decomposition of $A$ the rank one approximation matrix
$$
B=\begin{pmatrix}\ 0.024804 & -0.116881 \\ -0.116881 &\ 0.550767\end{pmatrix},
$$
then $\sigma_{\min}(B)=\delta^*$ and $\det(A-B)=0$.
\end{example}

In some cases, however, it is not necessary to inspect all principal components.
For special orthogonally invariant matrix norms, only the largest one is sufficient.

\begin{theorem}\label{thmPmaceRadMmat}
Suppose $A$ is an M-matrix. For the spectral or Frobenius norm we have 
$$
\delta^*= \sigma_{\min}(A).
$$
\end{theorem}

\begin{proof}
Denote by the hat a principal submatrix, so that $\hat{A}$ is a principal submatrix of $A$, and $\widehat{(A^{-1})}$ analogously for $A^{-1}$. By \cite{JohSmi2011} we have $0\leq \hat{A}^{-1}\leq\widehat{(A^{-1})}$. Then by the monotonicity of the spectral norm on nonnegative matrices we have 
\begin{equation*}
\rr(\hat{A})=\sigma_{\min}(\hat{A})
=\sigma^{-1}_{\max}(\hat{A}^{-1})
\geq \sigma^{-1}_{\max}\left(\widehat{(A^{-1})}\right)
\geq \sigma^{-1}_{\max}(A^{-1})
=\sigma_{\min}(A).
\qedhere
\end{equation*}
\end{proof}

\begin{theorem}
Suppose $A$ is (symmetric) positive definite. For the spectral or Frobenius norm we have 
$$
\delta^*= \lambda_{\min}(A).
$$
\end{theorem}

\begin{proof}
By Corollary~\ref{corPmatRadSpectr} and positive definiteness of $A$ we have
$$
\delta^* = \min\{\lambda_{\min}(\hat{A})\mmid
  \hat{A}\mbox{ is a principal submatrix of }A\}.
$$
Due to the Cauchy interlacing property of eigenvalues \cite{HorJoh1985}, $\lambda_{\min}(A)\leq\lambda_{\min}(\hat{A})$, whence $\delta^*= \lambda_{\min}(A)$.
\end{proof}

\begin{theorem}
Suppose $A$ is an M-matrix. For the max-norm we have 
$$
\delta^*= \frac{1}{e^TA^{-1}e}.
$$
\end{theorem}

\begin{proof}
For the same reasons as in the proof of Theorem~\ref{thmPmaceRadMmat} we derive
\begin{equation*}
\rr(\hat{A})
=\frac{1}{\max_{y,z\in\{\pm1\}^m}y^T\hat{A}^{-1}z}
=\frac{1}{e^T\hat{A}^{-1}e}
\geq\frac{1}{e^T\widehat{(A^{-1})}e}
\geq\frac{1}{e^TA^{-1}e}
=\rr(A).
\qedhere
\end{equation*}
\end{proof}

\section{M-matrices}\label{sMmat}

Recall that $A\in\R^{n\times n}$ is an M-matrix if $a_{ij}\leq0$ for every $i\not=j$ and $A^{-1}\geq0$. The latter can be replaced by many other equivalent conditions, we particularly utilize the condition $Av>0$ for certain $v>0$.

M-matrices form an important sub-class of P-matrices and they also appear in optimization in other situations \cite{Fie2006,Jac1977}, including stability of  Leontief's input-output analysis in economic systems.

\paragraph{Parametrization -- simple bounds.} 
Consider the parametrization of an M-matrix $A$ in the form $A_{\delta}=A-\delta \tilde{A}$, where $\tilde{A}\in\R^{n\times n}$ is given.
In the pursuit of finding the range of values, for which $A_{\delta}$ remains an M-matrix, we can easily construct an inner estimation.
The condition that the off-diagonal entries of $A_{\delta}$ should be non-positive yields a set of linear constraints on $\delta$.
Let $v>0$ such that $Av>0$. Then $A_{\delta}v>0$ gives us another linear constraint, from which we obtain:

\begin{theorem}\label{thmMmatPar}
$A_{\delta}$ is an M-matrix for each $\delta\in(\unum{\delta},\onum{\delta})$, where
\begin{align*}
\unum{\delta}
 &= \max\left\{ \max_{i\not=j, \tilde{a}_{ij}>0}\frac{a_{ij}}{\tilde{a}_{ij}},\ 
   \max_{i: \tilde{A}_{i*}v<0} \frac{A_{i*}v}{ \tilde{A}_{i*}v} \right\},\\
\onum{\delta}
&=\min\left\{ \min_{i\not=j, \tilde{a}_{ij}<0}\frac{a_{ij}}{\tilde{a}_{ij}},\ 
   \min_{i: \tilde{A}_{i*}v>0} \frac{A_{i*}v}{ \tilde{A}_{i*}v} \right\}.
\end{align*}
\end{theorem}

More precise estimation we get in the situation when $\tilde{A}$ has rank one, that is, it has the form of $\tilde{A}=ab^T$ for some $a,b\in\R^n$. In this case we have:

\begin{theorem}\label{thmMmatParRankone}
If $\tilde{A}=ab^T$, then the set of admissible values of $\delta$ for $A_{\delta}$ to be an M-matrix is described by linear constraints
\begin{align*}
\delta a_ib_j&\geq a_{ij},\quad \forall i\not=j,\\
\delta\left( (b^TA^{-1}a) A^{-1}- A^{-1}ab^TA^{-1}\right)&\leq A^{-1},\\
\delta b^TA^{-1}a&<1.
\end{align*}
\end{theorem}

\begin{proof}
We have by the Sherman--Morrison formula 
\begin{align*}
A_{\delta}^{-1}
=(A-\delta ab^T)^{-1}
=A^{-1}+\frac{\delta}{1-\delta b^TA^{-1}a}A^{-1}ab^TA^{-1}.
\end{align*}
Thus we have $A_{\delta}^{-1}\geq0$ iff $(1-\delta b^TA^{-1}a)A^{-1}+\delta A^{-1}ab^TA^{-1}\geq0$ and the denominator $1-\delta b^TA^{-1}a$ is positive.
\end{proof}

From the linear constraints one easily derives the interval range of admissible values. The interval may be closed, open or semi-open since the description of the admissible values for $\delta$ contains both strict and non-strict inequalities.

\paragraph{Parametrization -- exact bounds.} 
By definition, $A_{\delta}$ is an M-matrix if the offdiagonal entries are non-positive and the real eigenvalues (as well as real parts of all eigenvalues) are nonnegative. The former is handled easily, and for the latter we simply control that $\det(A_{\delta})$ stays positive. Due to continuity of eigenvalues, this implies that all real eigenvalues (and real parts of all eigenvalues) stay positive. Therefore, we handle this problem by the techniques from Section~\ref{sDet}.

\paragraph{M-matrix radius.}
For a given matrix norm, define the radius of M-matrix property of $A$ as
\begin{align*}
\delta^*:=  \sup\{\delta\geq0 \mmid 
  A+A'\mbox{ is an M-matrix }\forall A': \|A'\|<\delta\}.
\end{align*}
Notice the fundamental difference to P-matrix radius of an M-matrix discussed in Theorem~\ref{thmPmaceRadMmat}. For example, for $A=I_n$, the P-matrix radius is $1$, but the M-matrix radius is $0$ as an arbitrarily small perturbation (e.g., increase of the off-diagonal entries) might violate M-matrix property.

Considering the max-norm, we are seeking for maximal $\delta\geq0$ such that all matrices inside the interval matrix $[A-\delta E,A+\delta E]$ are M-matrices. Due to monotonicity, the worst case is the matrix $A-\delta E$, so the radius $\delta^*$ is found by the parametrization discussed above. We have also to incorporate the condition on nonpositivity of off-diagonal entries, that is, the restriction $\delta\leq -\max_{i\not=j}a_{ij}$.

Consider now a more general class of norms. We introduce first some simple lower and upper bounds on $\delta^*$; we will build on them also in the next section.

\begin{theorem}\label{thmMmatRadLower}
Let $k>0$ be large enough. 
Then for every consistent matrix norm satisfying \nref{normAss} we have 
$$\textstyle
\delta^*\geq
 \min\left\{ k-\|kI_n-A\|,\ -\max_{i\not=j}a_{ij}\right\}.
$$
\end{theorem}

\begin{proof}
Let $A'$ be a perturbation matrix. Then $A+A'$ is an M-matrix iff the offdiagonal entries are non-positive and for some $k$ it holds
\begin{align}\label{MmatRadEstim}
kI_n-A-A'\geq0,\ \ \rho(kI_n-A-A')<k.
\end{align}
In view of \nref{normAss}, for any $i\not=j$, we have $(A+A')_{ij}\leq0$ whenever $\|A'\|\leq -a_{ij}$.

Since 
\begin{align}\label{MmatRadCons}
\rho(kI_n-A-A')
\leq \|kI_n-A-A'\|
\leq \|kI_n-A\|+\|A'\|,  
\end{align}
the condition \nref{MmatRadEstim} is satisfied provided $\|A'\|\leq k-\|kI_n-A\|$.
\end{proof}

\begin{theorem}\label{thmMmatRadUpper}
For every consistent matrix norm satisfying \nref{normI} and \nref{normAss2} we have 
$$\textstyle
\delta^*\leq
 \min\left\{ k-\rho(kI_n-A),\ -\max_{i\not=j}a_{ij}\right\}.
$$
\end{theorem}

\begin{proof}
Consider the perturbation matrix in the form $A'=\beta I_n$, whence $\|A'\|=\beta$ by \nref{normI}.
The condition $\rho(kI_n-A+\beta I_n)<k$ holds if and only if $\rho(kI_n-A)+\beta <k$, yielding the upper bound $\beta< k-\rho(kI_n-A)$. Therefore also  $\delta^*\leq k-\rho(kI_n-A)$.

Now, consider the perturbation matrix in the form $A':=(\eps-a_{ij})e_i^{}e_j^T$, where $i\not=j$ and $\eps>0$ is arbitrarily small. Then $\|A'\|=\eps-a_{ij}$ by \nref{normAss2}. Since $(A+A')_{ij}=\eps$, it must $\delta^*<\eps-a_{ij}$. This holds for any $\eps>0$, so we can conclude $\delta^*\leq-a_{ij}$.
\end{proof}

\begin{remark}\label{remKvalue}
What is a sufficient value for $k$? The value of $k=2\max_i a_{ii}$ is sufficient as every perturbation matrix $A'$ satisfies $|a'|_{ii}\leq a_{ii}$, whence $kI_n-(A+A')\geq0$. Therefore the spectral radius of $kI_n-(A+A')$ is its largest eigenvalue and for all $k'>k$ we have $\rho(k'I_n-A-A')=\rho(kI_n-A-A')+k'-k$.
\end{remark}

Comparing the lower and upper bounds from Theorems~\ref{thmMmatRadLower} and~\ref{thmMmatRadUpper}, we see that the bounds are tight provided the matrix norm does not much overestimate the spectral radius.
If, for example, we employ the spectral norm and $A$ is a symmetric M-matrix, then both bounds are identical, yielding the exact value of $\delta^*$.

As for the parametrization, the exact M-matrix radius can be obtained by the reduction to sign stability of the determinant of the whole matrix.

\begin{theorem}\label{thmMmatRad}
For every matrix norm satisfying \nref{normAss} we have 
\begin{align*}\textstyle
\delta^*=\min
 \left\{\min_{i\not=j}\{-a_{ij}\},\rr(A)\right\}.
\end{align*}
\end{theorem}

\begin{proof}
Denote by $\delta^*_1$ the radius of non-positivity of the offdiagonal entries. Then $\delta^*_1=\min_{i\not=j}\{-a_{ij}\}$. For $i\not=j$ and the perturbation matrix $A':=(\eps-a_{ij})e_i^{}e_{j}^T$, $\eps>0$, we have $(A+A')_{ij}=\eps$. By \nref{normAss2} we also have $\|A'\|=\eps-a_{ij}$, so $\delta^*_1\leq -a_{ij}$. On the other hand, for every perturbation matrix $A'$ such that $\|A'\|\leq-a_{ij}$ we have by \nref{normAss1} $a'_{ij}\leq -a_{ij}$, so $(A+A')_{ij}\leq0$.

The formula $\delta^*=\min\{\delta^*_1,\rr(A)\}$ now follows from the fact that it is sufficient to keep the determinant of $A+A'$ positive since then all eigenvalues also remain positive and thus $A+A'$ stays to be an M-matrix.
\end{proof}

\begin{corollary}
For the spectral or Frobenius norm, we have
\begin{align*}\textstyle
\delta^*=\min_{i\not=j}\{-a_{ij},\sigma_{\min}(A)\}.
\end{align*}
\end{corollary}

\begin{example}\label{exMmatRad}
Consider the matrix $A$ and the spectral norm,
$$
A=\begin{pmatrix}10&-2\\-1&10\end{pmatrix}.
$$
Considering $A$ as an M-matrix, the lower bound (Theorem~\ref{thmMmatRadLower}),  the upper bound (Theorem~\ref{thmMmatRadUpper}) and the exact value (Theorem~\ref{thmMmatRad}) are the same, $\delta^*_M=1$. This is due to the $(2,1)$-entry of the matrix, which needs to be nonpositive.

Considering the matrix as a P-matrix, we obtain a larger radius. The P-matrix radius is $\delta^*_P\approx8.5125$, attained as the 2-norm of the whole matrix.

Let now
$$
B=\begin{pmatrix}20&-12\\-11&20\end{pmatrix}.
$$
The lower bound, the upper bound and the exact value of the M-matrix radius are $8.4938\leq \delta^*_M=8.5062\leq 8.5109$. We see that both lower and upper bounds are quite tight.
Considering this matrix as a P-matrix, the corresponding P-matrix radius has the same value of $\delta^*_P=8.5062$. Therefore, the nearest matrix to $B$ that is not an M-matrix is also not a P-matrix.
\end{example}

\section{H-matrices}\label{sHmat}

Recall that a matrix $A\in\R^{n\times n}$ is called an \emph{H-matrix} if the so called \emph{comparison matrix} $\langle A\rangle$ is an M-matrix, where $\langle A\rangle_{ii}=|a_{ii}|$ and $\langle A\rangle_{ij}=-|a_{ij}|$ for $i\not=j$.
H-matrices with positive diagonal represent a large class of efficiently verifiable P-matrices. 

\paragraph{Parametrization.}
Consider the parametrization of an H-matrix $A$ in the form $A_{\delta}=A-\delta \tilde{A}$, where $\tilde{A}\in\R^{n\times n}$ is given. For a sufficiently small $\delta$, the comparison matrix of $A_{\delta}$ draws $\langle A-\delta \tilde{A}\rangle=A'-\delta \tilde{A}'$ and both matrices $A'$ and $\tilde{A}'$ are trivially derived. Thus the problem of determining the set of admissible values of $\delta$, for which $A_{\delta}$ remains an H-matrix, directly reduces to the previous case of M-matrices. In this way, adaptation of Theorem~\ref{thmMmatPar} for a lower bound on the maximal admissible value $\delta^*$ takes the following form.

\begin{theorem}
Let $v>0$ such that $\langle A\rangle>0$. Then $A_{\delta}$ is an H-matrix for each $\delta\in(\unum{\delta},\onum{\delta})$, where
\begin{align*}
\unum{\delta} &= 
 \max\left\{ \max_{i\not=j, \tilde{a}'_{ij}>0}\frac{a'_{ij}}{\tilde{a}'_{ij}},\ 
 \max_{ i: \tilde{A}'_{i*}v<0} \frac{A'_{i*}v}{ \tilde{A}'_{i*}v} \right\},\\
\onum{\delta} &= 
 \min\left\{ \min_{i\not=j, \tilde{a}'_{ij}<0}\frac{a'_{ij}}{\tilde{a}'_{ij}},\  
 \min_{i: \tilde{A}'_{i*}v>0} \frac{A'_{i*}v}{ \tilde{A}'_{i*}v}\right\}.
\end{align*}
\end{theorem}

Notice that the bound of $\unum{\delta}$ or $\onum{\delta}$ may be caused due to the change of the sign of an offdiagonal entry. This is a restriction for M-matrices, but not for H-matrices. Therefore, it makes sense to put $\delta:=\unum{\delta}$ or $\delta:=\onum{\delta}$ and to repeat this process again with updated $\tilde{A}'$. Since the change of the sign may happen only once for each offdiagonal entry, there are possible at most $n^2-n$ iterations.

\begin{example}\label{exHmatPar}
Let
$$
A=\begin{pmatrix}10&-2\\-1&10\end{pmatrix},\quad
\tilde{A}=\begin{pmatrix}1&1\\1&1\end{pmatrix}.
$$

\emph{M-matrix parametrization. }
The matrix $A$ is an M-matrix, which can be confirmed by verifying $Av>0$ for $v=(1,1)^T$, for example. Applying the method for computing an interval of admissible values for the M-matrix property, we calculate $[\unum{\delta},\onum{\delta}]=[-1,4]$. We can consider the closed interval since the endpoints are attained. 

By using the second method via determinants, by obtain tighter (exact) bounds. The condition $\det(A_{\delta})>0$ leads to the bound $\delta<\frac{98}{23}\approx4.26$, and the other conditions $\det(A^{ij}_{\delta})>0$ do not change it. The nonpositivity of the offdiagonal entries yield the lower bound $\delta\geq-1$. In total, the range of admissible values is $[-1,\frac{98}{23})$.

If we take into account that $\tilde{A}$ has rank one and utilizing Theorem~\ref{thmMmatParRankone}, then we have the interval of admissible values $[\unum{\delta},\onum{\delta})=[-1,\frac{98}{23})$ in the first step. The right end-point $\frac{98}{23}\approx4.26$  is now optimal, and it is not attained.

\emph{H-matrix parametrization. }
If we consider the matrix $A$ from the above example as an H-matrix, we can derive a larger interval of admissible values. In the first iteration, we arrive at the same interval $[\unum{\delta},\onum{\delta}]=[-1,4]$. For $\delta=4$, no improvement happens, but for $\delta=-1$, we obtain a new interval of admissible values $[-2,4]$. For $\delta=-2$, we get a resulting interval $(-\infty,4]$. Thus, $A_{\delta}$ is an H-matrix for all $\delta\in(-\infty,4]$. Due to the heuristic nature of the method, this interval is not optimal. 
\end{example}

\paragraph{H-matrix radius.}
For a given matrix norm, define the radius of H-matrix property of $A$ as
\begin{align*}
\delta^*:=  \sup\{\delta\geq0 \mmid 
  A+A'\mbox{ is an H-matrix }\forall A': \|A'\|<\delta\}.
\end{align*}

Concerning the max-norm, $\delta^*$ is simply found by the parametrization applied to the parametric matrix  $\langle A\rangle-\delta E$.

We present various lower and upper bounds for various matrix norms.
The value of $k$ is large enough; see Remark~\ref{remKvalue}.

\begin{theorem}\label{thmHmatRadLowerAbs}
For every consistent monotone absolute matrix norm we have 
$$\textstyle
\delta^* \geq k-\| kI_n-\langle A\rangle \|.
$$
\end{theorem}

\begin{proof}
Let $A'$ be a perturbation matrix. In order that $\langle A+ A'\rangle=\langle A\rangle+\tilde{A}'$ is an M-matrix, it must hold
\begin{align}\label{ineqPfThmHmatRadCOnsMonAbs}
kI_n-\langle A\rangle-\tilde{A}'\geq0,\ \ \rho(kI_n-\langle A\rangle-\tilde{A}')<k.
\end{align}
Since 
\begin{align*}
\rho(kI_n-\langle A\rangle-\tilde{A}')
&\leq \|kI_n-\langle A\rangle-\tilde{A}'\|
 \leq  \|kI_n-\langle A\rangle\|+\||\tilde{A}'|\|\\
&\leq \|kI_n-\langle A\rangle\|+\||A'|\|, 
\end{align*}
the condition \nref{ineqPfThmHmatRadCOnsMonAbs} is satisfied as long as $\|A'\|=\||A'|\|< k-\|kI_n-\langle A\rangle\|$.
\end{proof}

Since the spectral norm is not covered in the above statement, we derive a lower bound separately.

\begin{theorem}\label{thmHmatRadLower}
Let $\|\cdot\|$ be the Frobenius or the induced 1-~or $\infty$-norm. Then the H-matrix radius for the spectral norm satisfies
$$
\delta^*
 \geq \frac{1}{\sqrt{n}}(k-\| kI_n-\langle A\rangle \|)
 \geq \frac{k}{\sqrt{n}}-\| kI_n-\langle A\rangle \|_2.
$$
\end{theorem}

\begin{proof}
Let $A'$ be a perturbation matrix. In order that $\langle A+ A'\rangle=\langle A\rangle+\tilde{A}'$ is an M-matrix, it must hold
\begin{align}\label{ineqPfThmHmatRadCOnsMon}
kI_n-\langle A\rangle-\tilde{A}'\geq0,\ \ \rho(kI_n-\langle A\rangle-\tilde{A}')<k.
\end{align}
Since 
\begin{align*}
\rho(kI_n-\langle A\rangle-\tilde{A}')
&\leq \|kI_n-\langle A\rangle-\tilde{A}'\|
 \leq \|kI_n-\langle A\rangle\|+\|\tilde{A}'\|\\
& \leq \|kI_n-\langle A\rangle\|+\|A'\|
 \leq \|kI_n-\langle A\rangle\|+\sqrt{n}\|A'\|_2, 
\end{align*}
the condition \nref{ineqPfThmHmatRadCOnsMon} is satisfied whenever $\sqrt{n}\|A'\|_2< k-\|kI_n-\langle A\rangle\|$.
\end{proof}

\begin{theorem}\label{thmHmatRadUpper}
For every consistent matrix norm satisfying \nref{normI} we have
$$\textstyle
\delta^*\leq  k-\rho(kI_n-\langle A\rangle).
$$
\end{theorem}

\begin{proof}
Consider the perturbation matrix in the form $A'=\beta I_n$, whence $\|A'\|=\beta$.
The condition $\rho(kI_n-\langle A\rangle+\beta I_n)<k$ holds as long as $\rho(kI_n-\langle A\rangle)+\beta <k$, yielding the upper bound $\beta< k-\rho(kI_n-\langle A\rangle)$. Therefore also  $\delta^*\leq k-\rho(kI_n-\langle A\rangle)$.
\end{proof}


\begin{theorem}\label{thmHmatRadRR}
For any matrix norm we have $\delta^* \leq \rr(A)$.
For any monotone absolute matrix norm we have $\delta^* \geq \rr(\langle A\rangle)$.
\end{theorem}

\begin{proof}
The first statement is obvious as any H-matrix is nonsingular.

To show the second one, let $A'$ be any such that $\|A'\|<\rr(\langle A\rangle)$. Then $\langle A+A'\rangle=\langle A\rangle+\tilde{A}'$, where $\tilde{A}'$ is such that $|\tilde{A}'|\leq|A'|$. Thus $\|\tilde{A}'\|=\||\tilde{A}'|\|\leq\||A'|\|=\|A'\|<\rr(\langle A\rangle)$, from which $\langle A+A'\rangle$ is nonsingular. Due to continuity of eigenvalues it must have all real eigenvalues positive, and therefore be an M-matrix, as otherwise for a smaller perturbation (in the same direction) it would be singular.
\end{proof}

\begin{example}
Consider the same matrices as in Example~\ref{exMmatRad} and the spectral norm. 
$$
A=\begin{pmatrix}10&-2\\-1&10\end{pmatrix}.
$$
Considering $A$ as an H-matrix, the corresponding lower bound (Theorem~\ref{thmHmatRadLower} with induced 1-norm) and the upper bound (Theorem~\ref{thmHmatRadUpper}) on the H-matrix radius are $5.6569\leq \delta^*_H\leq 8.5858$.
The upper bound from Theorem~\ref{thmHmatRadRR} improves the estimation to $\delta^*_H\leq 8.5125$.
Considering the matrix as a P-matrix, the P-matrix radius would be the same $\delta^*_P\approx8.5125$. 
Notice that always $\delta^*_H\leq\delta^*_P$ holds. 

Consider now the second matrix
$$
B=\begin{pmatrix}20&-12\\-11&20\end{pmatrix}.
$$
The lower and upper bounds on the H-matrix radius are $5.6569\leq \delta^*_H \leq 8.5109$. 
The upper bound from Theorem~\ref{thmHmatRadRR} is again better, yielding $\delta^*_H\leq 8.5062$.
If we consider $B$ as a P-matrix, then the corresponding P-matrix radius has the same value of $\delta^*_P=8.5062$. 

Consider now a novel matrix
$$
C=\begin{pmatrix}2&1&1\\1&2&1\\1&1&2\end{pmatrix}.
$$
The lower and upper bounds on the H-matrix radius are $0\leq \delta^*_H \leq 0$, hence the H-matrix radius is zero. 
In contrast, the P-matrix radius has the value of $\delta^*=1$. Therefore, the P-matrix radius can be considerably larger than the H-matrix radius.

Eventually, consider the identity matrix $I_2$. The lower bound and the upper bounds on the H-matrix radius are $0.70711\leq \delta^*_H \leq 1$. In this case, the lower bound is tight since the perturbation matrix
$$
A'=\begin{pmatrix}0.5 & -0.5 \\ 0.5 & 0.5\end{pmatrix}.
$$
has the spectral norm $0.70711$ and $I_2-A'$ passes to be an H-matrix. Therefore $\delta^*_H\approx 0.70711$.
\end{example}

\section{Totally positive matrices}\label{sTp}

A matrix $A\in\R^{n\times n}$ is called \emph{totally positive} if the determinants of all submatrices are positive. In spite of the huge number of determinants in definition, a  suitable selection of $n^2$ of them is indeed necessary, which makes the problem tractable; see Fallat and Johnson \cite{FalJoh2011}. These matrices are called initial submatrices, and they are defined as follows: They are the submatrices the rows of which are indexed by $I$ and columns by $J$, and these index sets have the form $I=\seznam{k}$, $J=\sseznam{\ell}{\ell+k-1}$ or vice versa.
Let us denote the initial submatrices as $A^{(1)},\dots,A^{(n^2)}$.
Totally positive matrices thus obviously make another efficiently recognizable subclass of P-matrices.

\paragraph{Parametrization.}
Let $A_{\delta}=A-\delta \tilde{A}$ be a parametrization of the matrix $A$ and we want to compute the range of the values of $\delta$ for which  $A_{\delta}$ remains totally positive. On account of Section~\ref{sDet}, this problem just reduces to $n^2$ cases of parametrization of positive determinants. Thus, we compute the ranges of admissible values for $\det(A^{(i)}_{\delta})>0$, $i=1,\dots,n^2$, and take intersection of them.

Notice that a parametrization of one entry of $A$ was already discussed, e.g., in Fallat and Johnson \cite{FalJoh2011}.

\paragraph{Totally positive radius.}
The radius of totally positivity is
\begin{align*}
\delta^*:=  \sup\{\delta\geq0 \mmid 
  A+A'\mbox{ is  totally positive }\forall A': \|A'\|<\delta\}.
\end{align*}
This radius is simply determined based on Theorem~\ref{thmDetRad}.

\begin{theorem}\label{thmTpRad}
We have $\displaystyle\delta^*=\min_{i=1,\dots,n^2}r(A^{(i)})$.
\end{theorem}

In particular, for the spectral or Frobenius norm we have $\delta^*=\min_{i=1,\dots,n^2}\sigma_{\min}(A^{(i)})$, so the radius is computable in polynomial time.

Consider now the max-norm. Theorem~\ref{thmTpRad} is not convenient since computing the regularity radius is NP-hard problem. Nevertheless, $\delta^*$ can be still determined efficiently. The radius $\delta^*$ can be equivalently formulated as the supremal $\delta$ such that the interval matrix $[A-\delta ee^T,A+\delta ee^T]$ contains totally positive matrices only. By Garloff \cite{Gar1982}, this is equivalent the checking totally positivity of two matrices $A-\delta ss^T$ and $A+\delta ss^T$ only, where $s:=(1,-1,1,-1,\dots)^T$. Therefore, we compute the largest interval $(\delta^1,\delta^2)\ni0$ of admissible values for the parametrized matrix $A_{\delta}=A-\delta ss^T$ and put $\delta^*:=\min\{-\delta^1,\delta^2\}$. Notice that parametrization is simple in this case as $ss^T$ has rank one and we can proceed by Theorem~\ref{thmDetParRankone}.

\section{Inverse M-matrices}

A matrix $A\in\R^{n\times n}$ is \emph{an inverse M-matrix} if it is nonsingular and $A^{-1}$ is an M-matrix \cite{JohSmi2011}.

\paragraph{Parametrization.}
Particular perturbations and convex combination property were presented in \cite{JohSmi2011}.
As in the above cases, we consider the parametrized matrix $A$ in the form $A_{\delta}=A-\delta \tilde{A}$.
The matrix must remain nonsingular, which gives rise to the constraint $\det(A_{\delta})>0$. The condition $(A_{\delta}^{-1})_{ij}\leq0$ for $i\not=j$ equivalently reads $\frac{1}{\det(A_{\delta})}(-1)^{i+j}\det(A^{ji}_{\delta})\leq0$, or $(-1)^{i+j}\det(A^{ji}_{\delta})\leq0$. 
The last condition $A_{\delta}\geq0$ takes the form of linear constraints, but surprisingly needn't be considered. Due to continuity of eigenvalues, when $\det(A_{\delta})$ remains positive also the real eigenvalues of $A_{\delta}$ and $A_{\delta}^{-1}$ remain positive and thus $A_{\delta}^{-1}$ remains an M-matrix.
Therefore, in total our constraints are
\begin{align*}
\det(A_{\delta})>0,\quad
(-1)^{i+j}\det(A^{ji}_{\delta})\leq0,\ \ i\not=j.
\end{align*}
We will handle these constraints as in Section~\ref{sDet}. 

If in addition $\tilde{A}$ has rank one, then these constraints take a linear form. We can derive these constaints directly in a compact form.

\begin{theorem}
If $\tilde{A}=ab^T$, then the set of admissible values of $\delta$ for $A_{\delta}$ to be inverse M-matrix is described by linear constraints
\begin{align*}
\delta\left( (b^TA^{-1}a) A^{-1}- A^{-1}ab^TA^{-1}\right)_{ij}&\geq (A^{-1})_{ij},
\quad i\not=j\\
\delta b^TA^{-1}a&<1.
\end{align*}
\end{theorem}

\begin{proof}
The set of admissible values is characterized by $(A_{\delta}^{-1})_{ij}\leq0$, $i\not=j$. As in the proof of Theorem~\ref{thmMmatParRankone}, we use the Sherman--Morrison formula 
\begin{align*}
A_{\delta}^{-1}
=(A-\delta ab^T)^{-1}
=A^{-1}+\frac{\delta}{1-\delta b^TA^{-1}a}A^{-1}ab^TA^{-1},
\end{align*}
from which the resulting system is derived.
\end{proof}

\paragraph{Inverse M-matrix radius.}
For a given matrix norm, define the radius of the inverse M-matrix property of $A$ as
\begin{align*}
\delta^*:=  \sup\{\delta\geq0 \mmid 
  A+A'\mbox{ is an inverse M-matrix }\forall A': \|A'\|<\delta\}.
\end{align*}

For any matrix norm, computing the inverse M-matrix radius can be reduced to computing $n^2-n+1$ radii of nonsingularity in the same norm.
Thus, $\delta^*$ is efficiently computable as long as the radius of nonsingularity is efficiently computable. 

\begin{theorem}
We have 
$
\delta^* = \min_{i\not=j}\{\rr(A),\rr(A^{ij})\}.
$
\end{theorem}

\begin{proof}
Due to formula $(A^{-1})_{ij}=\det(A)^{-1}(-1)^{i+j}\det(A^{ji})$, in order that $A^{-1}$ has nonpositive offdiagonal entries, we must ensure that $\det(A)$ and $\det(A^{ji})$ remain sign stable. Sign stability of $\det(A)$ also implies that the real eigenvalues of $A$ and $A^{-1}$ remain positive and thus $A^{-1}$ remains an M-matrix. Due to Theorem~\ref{thmDetRad}, the radius of the sign stable determinant is equal to the radius of nonsingularity.
\end{proof}

\begin{corollary}
In the spectral or Frobenius norm, we have 
$
\delta^* = \min_{i\not=j}\{\sigma_{\min}(A),\sigma_{\min}(A^{ij})\}.
$
\end{corollary}

\section{Inverse nonnegative matrices}

A matrix $A\in\R^{n\times n}$ is \emph{inverse nonnegative} if $A^{-1}\geq0$. As a particular sub-class, M-matrices are inverse nonnegative.

\paragraph{Parametrization.}
Let $A_{\delta}=A-\delta \tilde{A}$ be a parametrization of the matrix $A$ and we are again interested in computation of the range of the values of $\delta$ for which  $A_{\delta}$ remains inverse nonnegative. For any $i,j$, we have $(A^{-1})_{ij}=\frac{1}{\det(A)}(-1)^{i+j}\det(A^{ji})$. 
Therefore even in this case, parametrization reduces to the cases studied in Section~\ref{sDet}. The range of admissible values is determined from the constraints $\det(A_{\delta})>0$, $(-1)^{i+j}\det(A^{ji}_{\delta})\geq0$, $\forall i,j$.

Provided $A'$ has rank one, a more effective method exists. In this case the matrix has the form of $A'=ab^T$ for some $a,b\in\R^n$. 

\begin{theorem}\label{thmInParRankone}
If $\tilde{A}=ab^T$, then the set of admissible values of $\delta$ for $A_{\delta}$ to be inverse nonnegative is described by linear constraints
\begin{align*}
\delta\left( (b^TA^{-1}a) A^{-1}- A^{-1}ab^TA^{-1}\right)&\leq A^{-1},\\
\delta b^TA^{-1}a&<1.
\end{align*}
\end{theorem}

\begin{proof}
The set of admissible values is characterized by $A_{\delta}^{-1}\geq0$. As in the proof of Theorem~\ref{thmMmatParRankone}, we use the Sherman--Morrison formula 
\begin{align*}
A_{\delta}^{-1}
=(A-\delta ab^T)^{-1}
=A^{-1}+\frac{\delta}{1-\delta b^TA^{-1}a}A^{-1}ab^TA^{-1}.
\end{align*}
Thus the condition $A_{\delta}^{-1}\geq0$ has the equivalent formulation as demanded.
\end{proof}

\paragraph{Inverse nonnegativity radius.}
The inverse nonnegativity radius of $A$ is naturally defined as
\begin{align*}
\delta^*:=  \sup\{\delta\geq0 \mmid 
  A+A'\mbox{ is inverse nonnegative }\forall A': \|A'\|<\delta\}.
\end{align*}
By Theorem~\ref{thmDetRad} and the above observation regarding parametrization, we obtain:

\begin{theorem}\label{thmInvNonnegRad}
We have $\delta^*=\min_{i,j=1,\dots,n} \{ r(A), r(A^{ji}) \}$.
\end{theorem}

In particular, for the spectral or Frobenius norm we calculate the radius efficiently as $\delta^*=\min_{i,j=1,\dots,n} \{ \sigma_{\min}(A), \sigma_{\min}(A^{ji}) \}$. It may happen that there is a matrix $C$ in the $\delta^*$ distance from $A$ and such that it is not inverse nonnegative, and also it may happen that this distance is not achieved. That is, $\delta^*$ may be achieved as maximum or not. This is illustrated in the following example.

\begin{example}
Consider the spectral norm and the matrix 
$$
A=\begin{pmatrix}10&-1\\-1&10\end{pmatrix}.
$$
Then $\delta^*=\sigma_{\min}(A^{12})=1$ and $\sigma_{\min}(A)=9$. So every matrix no more far from $A$ than 1 is inverse nonnegative.

Now, let
$$
B=\begin{pmatrix}10&-9\\-9&10\end{pmatrix}.
$$
Then $\delta^*=\sigma_{\min}(B)=1$. The matrix
$$
C=\begin{pmatrix}10&-10\\-10&10\end{pmatrix}.
$$
is not inverse nonnegative, but its distance from $B$ is~$1$.
\end{example}

For the max-norm, Theorem~\ref{thmInvNonnegRad} is not convenient since computing the regularity radius is NP-hard problem. We can, however, determine $\delta^*$ efficiently. The radius $\delta^*$ can be equivalently formulated as the supremal $\delta$ such that the interval matrix $[A-\delta ee^T,A+\delta ee^T]$ contains inverse nonnegative matrices only. By the Kuttler theorem \cite{Kut1971} (for some extensions see \cite{Neu1990,Roh2012b}), this is equivalent the checking inverse nonnegativity of the lower and upper bound matrices $A-\delta ee^T$ and $A+\delta ee^T$. Therefore, we compute the largest interval $(\delta^1,\delta^2)\ni0$ of admissible values for the parametrized matrix $A_{\delta}=A-\delta ee^T$ and put $\delta^*:=\min\{-\delta^1,\delta^2\}$. Parametrization is again simple in this case since $ee^T$ has rank one and we can proceed by Theorem~\ref{thmInParRankone}.

\paragraph{Comparison.}
Let us now compare M-matrix and inverse nonnegativity radii. Let $A\in\R^{n\times n}$ be an M-matrix and denote by $\delta^*_{\text{M}}$ and $\delta^*_{\text{IN}}$ the M-matrix and inverse nonnegativity radii, respectively. In general, we have $\delta^*_{\text{M}}\leq\delta^*_{\text{IN}}$. Equality is attained for specific situations.

\begin{proposition}
Suppose the matrix norm satisfies \nref{normAss}. Then $\delta^*_{\text{M}}=\delta^*_{\text{IN}}$ provided
\begin{enumerate}[(i)]
\item
$n=2$, or
\item
$r(A)\leq-a_{ij}$ for all $i\not=j$.
\end{enumerate}
\end{proposition}

\begin{proof}
By Theorems~\ref{thmMmatRad} and~\ref{thmInvNonnegRad} and from the assumptions, we have
\begin{align*}
\delta^*_{\text{IN}}
= \min_{i,j=1,\dots,n} \{ r(A), |a|_{ij}) \}
\leq \min_{i\not=j} \{ r(A), |a|_{ij}) \}
= \delta^*_{\text{M}}
\end{align*}
when $n=2$, and we have
\begin{align*}
\delta^*_{\text{IN}}
= \min_{i,j=1,\dots,n} \{ r(A), r(A^{ji}) ) \}
\leq r(A)
= \min_{i\not=j} \{ r(A), -a_{ij}) \}
= \delta^*_{\text{M}}
\end{align*}
in the second case.
\end{proof}

When the assumptions are not satisfied, equality is not attained in general.

\section{Conclusion}

We investigated parametrization and stability of various matrix properties. In most of the cases, the maximal allowable perturbation coefficient can be determined by reduction to a polynomial number of simpler problems. Parametrization is particularly effective provided the perturbation matrix has rank one. The maximal circle of stability is also efficiently computable as long as we use a suitable matrix norm. The spectral and Frobenius norm are among the most convenient, whereas the max-norm cases can be intractable in some cases.

\subsubsection*{Acknowledgments.} 


\bibliographystyle{abbrv}
\bibliography{toler_mat_prop}

\end{document}